\documentclass[ejsv2, noshowframe]{imsart}
\RequirePackage[numbers, sort]{natbib}
\RequirePackage[colorlinks,citecolor=blue,urlcolor=blue]{hyperref}
\RequirePackage{graphicx, subcaption, comment}

\startlocaldefs
\theoremstyle{plain}

\newtheorem{theorem}{Theorem}[section]

\newtheorem{corollary}[theorem]{Corollary}
\newtheorem{proposition}[theorem]{Proposition}

\theoremstyle{definition}

\newtheorem*{example}{Example}

\theoremstyle{remark}

\newcommand{\norm}[1]{\left\lVert#1\right\rVert}

\newcommand{\e}{\epsilon}
\newcommand{\w}{\omega}

\newcommand{\tr}{\text{tr}}

\newcommand{\R}{\mathbb{R}}

\newcommand{\E}{\mathbb{E}}

\renewcommand{\P}{P}

\newcommand{\X}{\textbf{X}}

\newcommand{\B}{\mathcal{B}}

\newcommand{\C}{\mathcal{C}}

\let\temp\phi
\let\phi\varphi
\let\varphi\temp

\endlocaldefs

\begin{document}

\begin{frontmatter}
\title{Hyper-V uniform ergodicity of Markov chains}
\runtitle{Hyper-V uniform ergodicity of Markov chains}

\begin{aug}
\author[A]{\fnms{Austin}~\snm{Brown}\ead[label=e1]{austinbrown@tamu.edu}}
\and
\author[B]{\fnms{Kshitij}~\snm{Khare}\ead[label=e2]{kdkhare@stat.ufl.edu}}

\address[A]{Department of Statistics, Texas A\&M University, College Station, Texas, USA \printead[presep={,\ }]{e1}}
\address[B]{Department of Statistics, University of Florida, Gainesville, Florida, USA \printead[presep={,\ }]{e2}}
\runauthor{A. Brown and K. Khare}
\end{aug}

\begin{abstract}
We develop a new uniform drift condition and local minorization that implies a stronger weighted form of uniform ergodicity for Markov chains we call hyper-V uniform ergodicity. The convergence guarantees geometric decay of the bias towards the invariant measure independently of the initialization for all functions controlled by a dominating function $V$. A key advantage of the approach is that it bypasses the need to establish a global minorization condition, which is often substantially more difficult to verify in practice, while yielding stronger convergence guarantees than global minorization. 
Optimal convergence bounds in a minimax sense of the framework are established.
The utility of the framework is demonstrated through applications to the P\'olya--Gamma and Kolmogorov--Gamma Gibbs samplers. We also show qualitative hyper-V uniform ergodicity convergence for two-variable Gibbs samplers can be inferred by the form of the invariant measure, bypassing convergence analysis entirely.
\end{abstract}



\end{frontmatter}

\section{Introduction}

%Many Gibbs samplers \citep{geman:geman:1984} experience good performance in practice for modern statistical and AI applications in comparison to other algorithms. 

% We define and investigate a new class of Markov chains we call $V-$regularizing Markov chains.
% These Markov chains hyper-regularize a large class of functions implying a strong Lyapunov drift condition.
% We show that $V-$regularizing Markov chains often yield uniformly ergodicity.
% There are many theoretically advantages to uniformly ergodic Markov chains that provide practical benefits in simulations.
% The first advantage is the geometric convergence rate can be better than using a global minorization over the entire state space as the minorization only has to be taken over a small set which is often compact.
% Another advantage is the Markov chain is robust to the initialization.
% Another advantage is sub-Gaussian concentration through McDiarmid's inequality holds that is not available for geometrically ergodic Markov chains. 

Uniform ergodicity is a particularly strong form of geometric ergodicity that is rarely available in general-state Markov chains, but when it holds, it yields a range of powerful theoretical and practical advantages. In contrast to standard geometric ergodicity, uniform ergodicity provides total variation convergence rates that are independent of the initial state, leading to robustness with respect to initialization and eliminating the need for initialization considerations to determine burn-in. It also yields non-asymptotic guarantees for Markov chain Monte Carlo, with bounds that hold uniformly over the state space. Moreover, uniformly ergodic chains satisfy strong concentration properties: in particular, sub-Gaussian deviation inequalities, such as those derived via McDiarmid's inequality \citep{mcdiarmid:1989, marton:2003} in contrast to the weaker guarantees typically available under general geometric ergodicity.

Despite these advantages, the classical conditions used to establish uniform ergodicity are often restrictive and difficult to verify in practice. The standard approach relies on a global minorization (Doeblin) condition, requiring that the transition kernel $P$ of the Markov chain admits a uniform lower bound of the form 
\[
\inf_{x \in \X} P(x,\cdot) \ge \alpha \nu(\cdot),
\]
for some $\alpha \in (0, 1]$ and probability measure $\nu$ on the state space $\X$. This condition enforces a uniform mixing mechanism across the entire state space, ensuring that the chain loses memory of its initial condition in a single step with fixed probability. Directly verifying a global minorization condition generally requires detailed, uniform control of the transition kernel, which is particularly challenging for Markov chains on unbounded or high-dimensional spaces, or when the transition dynamics vary significantly across the state space \citep{mengersen:tweedie:1996, ChoiHobert2013}. Furthermore, while such conditions yield uniform convergence in total variation, they primarily control expectations of bounded test functions and do not, in general, provide direct guarantees for broader classes of possibly unbounded functions that are often of interest in applications. 

In this work, we replace the classical global minorization condition with a pair of structurally simpler assumptions, consisting of a \emph{uniform drift condition} and a \emph{local minorization condition}. The uniform drift condition (with respect to a function $V: \X \rightarrow [0, \infty)$) takes the form of a one-step uniform Lyapunov bound $\int V(y) P(x, dy) \le K$ for every $x \in \X$, which ensures that the chain reaches a suitably large sublevel set of $V$ with uniformly positive probability from any initial state. The minorization condition is then imposed only on this sublevel set, rather than globally.

These assumptions are typically {\em much easier to verify in practice}, as they require only local control of the transition kernel. Nevertheless, they are sufficient to recover a global mixing mechanism: we show that they imply uniform ergodicity with explicit rates. Moreover, the resulting bounds extend beyond total variation, yielding uniform control of expectations for functions dominated by $V$, and thus providing a \emph{stronger weighted form of uniform ergodicity}, which we term \emph{hyper-$V$ uniform ergodicity} (see Theorem~\ref{thm:main}). The resulting bounds are simple, explicit, and directly interpretable in terms of the drift and minorization parameters. 
While such results can, in principle, be obtained by specializing Theorem~1.3 of Hairer and Mattingly \citep{hairer:mattingly:2011} to the present setting, this typically requires additional technical arguments and the conclusions are not guaranteed to be optimal (see the discussion and comparison following Theorem \ref{thm:main}). In Section~\ref{sec:optimal:rates}, we refine these bounds further by developing optimal hyper-$V$ uniform ergodicity rates under the same assumptions (see Theorem~\ref{thm:convergence_sharp}). 
%In contrast, our approach yields these bounds via a more direct and transparent argument.
%In contrast, Section~\ref{sec:optimal:rates} develops optimal convergence bounds in a minimax sense within the present framework. 

Furthermore, this framework naturally extends to uniform convergence in relevant Wasserstein distances (see Corollary~\ref{cor:wasserstein}). By relating the underlying metric to the Lyapunov function $V$, the weighted-$V$ bounds can be transferred to transport-based distances without imposing additional regularity or contractivity assumptions on the transition kernel. This stands in contrast to existing Wasserstein ergodicity results (\cite{hairer:mattingly:2011, HairerMattinglyScheutzow2011, Eberle2016}), which typically rely on explicit contraction properties, and highlights the robustness of the present approach. 

We demonstrate the efficacy and versatility of this approach through several applications to practically relevant Markov chain Monte Carlo algorithms on general state spaces. First, we consider the Pólya--Gamma Gibbs sampler for Bayesian logistic regression introduced by Polson, Scott, and Windle \cite{PolsonScottWindle2013}, which was shown by Choi and Hobert \cite{ChoiHobert2013} to be uniformly ergodic in total variation via a global minorization argument. Second, we study the Kolmogorov--Gamma sampler for regression with continuous proportion data proposed by Lee and co-authors \cite{LeeEtAlKG}, for which uniform ergodicity in total variation was again established using global minorization techniques. For both samplers, we show that the associated Markov transition kernels satisfy a uniform $V$-drift condition together with a local minorization condition, where $V$ is chosen to control the squared $\ell_2$ norm of the regression coefficients (see Theorems \ref{thm:pg_convergence} and \ref{thm:kg_convergence}). Notably, these conditions can be verified with substantially less effort than establishing the global minorizations of \cite{ChoiHobert2013} and \cite{LeeEtAlKG}, as they require only local control of the transition kernel. More importantly, as a consequence, we obtain a strictly stronger form of convergence, namely \emph{hyper-$V$ uniform ergodicity}, which provides uniform control over a class of unbounded functions. In particular, this leads to uniform convergence to stationarity in the $\ell_2$-Wasserstein distance. To the best of our knowledge, such results have not been previously established for these models. 
Next, we show that there are some special cases where the new framework has limitations, such as in the independence Metropolis-Hastings sampler \cite{hastings:1970, tierney:1994, RobertCasella2004} where a global minorization is necessary but the new uniform drift condition improves existing convergence results.
%Next, we show that our framework also applies to the independence Metropolis-Hastings sampler \cite{hastings:1970, tierney:1994, RobertCasella2004} under suitable regularity conditions, further illustrating the breadth and practical relevance of the proposed methodology. 
Finally, we show qualitative hyper-$V$ uniform convergence can be inferred for two-variable Gibbs samplers if the posterior exhibits a specific form providing a practical tool bypassing convergence analysis. 

The paper is organized as follows.
Section~\ref{section:general_results} develops hyper-$V$ uniform ergodicity based on a uniform drift condition and local minorization condition, yielding explicit convergence bounds in both total variation and weighted-$V$ metrics. Section~\ref{sec:optimal:rates} refines these bounds by deriving optimal convergence rates within the proposed framework and extends the analysis to Wasserstein distances. 
Section~\ref{section:examples} investigates applications to explicit uniform convergence rates for the Polya-Gamma sampler (Section \ref{example:pg_sampler}), the Kolmogorov-Gamma sampler (Section \ref{example:kg_sampler}), and the Independence Metropolis-Hastings sampler (Section \ref{example:imh}). Qualitative convergence results are further developed based on the form of the target measure constructed from data augmentation (Section~\ref{example:general_data_augmentation}). Section~\ref{section:conclusion} discusses the results and future research directions.

\section{General result on hyper-V uniform ergodicity}
\label{section:general_results}

Let $\X$ be a Borel measurable space and $\B(\X)$ its Borel sigma-algebra.
Let $\P(\cdot, \cdot)$ be a Markov transition kernel on $\X \times \B(\X)$.
For Borel functions $f : \X \to \R$ and Borel probability measures $\nu$ on $\X$, define
\begin{align*}
&\P f(\cdot) = \int_{\X} f(y) \P(\cdot, dy), 
&&\nu \P( \cdot ) = \int_{\X} \P(y, \cdot) \nu(dy).
\end{align*}
For any initial Borel probability measure $\mu$ on $\X$, this defines a Markov chain $(X^{(t)})_{t = 0}^{\infty}$ on $\X$ where for integers $t \ge 1$, the conditionals $X^{(t)} \mid X^{(t-1)} = x$ have the distribution $\P(x, \cdot )$ for $x \in \X$.
Define $\P^1 = \P$ and for integers $t \ge 1$, $\P^{t + 1}f(\cdot) = \P \P^t f(\cdot)$.

We will assume there exists a unique invariant probability measure $\Pi$ on $\X$ meaning $\Pi P = \Pi$.
One of the main interests for Markov chains is controlling the bias of $\P^t f(\cdot)$ to $\int_{\X} f d\Pi$ over a class of functions $f : \X \to \R$.
Define the total variation distance by
\[
\norm{\P^t(x, \cdot) - \Pi(\cdot)}_{\mathrm{TV}}
= \sup_{ f : \X \to [0, 1]} \left| \P^tf(x) - \int_{\X} f d\Pi \right|.
\]
A Markov chain is uniformly ergodic if there is a $\rho \in (0, 1)$ and $M \in (0, \infty)$ such that for all $x \in \X$ and for all integers $t \ge 0$
\[
\norm{\P^t(x, \cdot) - \Pi(\cdot)}_{\mathrm{TV}}
\le \rho^t M.
\]
In particular, independent of the initialization, the Markov chain bias of $\P^t f (\cdot)$ geometrically contracts towards $\int_{\X} f d\Pi$ uniformly over any Borel measurable function $f : \X \to [0, 1]$.

In this section, we establish a form of uniform ergodicity that goes beyond classical total variation convergence. The key ingredients are a uniform drift condition (in contrast to the usual geometric drift) together with a local minorization condition, which together induce a two-step global minorization. This mechanism yields uniform exponential convergence in total variation as well as uniform control over expectations of functions growing at rate $V$, giving rise to what we term \emph{hyper-$V$ uniform ergodicity}.
For a Borel function $W : \X \to [0, \infty)$, define a weighted total variation distance by
\[
\norm{\P^t(x, \cdot) - \Pi(\cdot)}_{W}
= \sup_{ |f| \le W} \left| \P^t f(x) - \int_{\X} f d\Pi \right|
= \int_{\X} W(y) | \P^t(x, dy) - \Pi(dy)|
\]
where the supremum is taken over Borel functions $f : \X \to \R$ such that $|f| \le W$.
With $W = 1$, this definition is twice the standard total variation as we have defined it. 
The following shows convergence in this weighted distance uniform in the initialization for the two-step chain.

\begin{theorem}[Hyper-$V$ uniform ergodicity] \label{thm:main}
Let $V : \X \to [0, \infty)$ be a Borel measurable function.
Assume the following hold:
\begin{enumerate}
\item
There exists a constant $K \in (0,\infty)$ such that
\begin{align}
P V(x) \le K, \qquad \forall x \in \X.
\label{eq:drift}
\end{align}

\item
There exists a constant $r > 1$, a probability measure $\nu$ on $\X$, and a constant $\alpha_r \in (0, 1]$ such that
\begin{align}
\inf_{\{x \in \X : V(x) \le rK\}} P(x, \cdot) \;\ge\; \alpha_r \, \nu(\cdot).
\label{eq:local_minorization}
\end{align}
\end{enumerate}
Let $c : [0, \infty) \to [0, \infty)$ be any nondecreasing concave function.
Then for all $x \in \X$ and all integers $t \ge 1$, the following hold:
\begin{align*}
\left\| P^{2t}(x, \cdot) - \Pi(\cdot) \right\|_{\mathrm{TV}}
&\le \big(1 - \alpha_r (1 - 1/r)\big)^t, \\[6pt]
\left\| P^{2t + 1}(x, \cdot) - \Pi(\cdot) \right\|_{c(V)}
%\sup_{f : \X \to \R,\; |f| \le c( V )}
%\left|
%P^{2t+1}f(x) - \int_{\X} f \, d\Pi
%\right|
&\le 2 c(K) \, \big(1 - \alpha_r (1 - 1/r)\big)^t.
\end{align*}
\end{theorem}

\begin{proof}
By Markov's inequality,
\[
P(x, \{ V > rK \})
\le \frac{PV(x)}{rK}
\le \frac{1}{r},
\qquad \forall x \in \X.
\]

Let $B \subseteq \X$ be Borel and $x \in \X$. Then
\begin{align*}
P^2(x,B)
= \int_{\X} P(x',B)\, P(x,dx') 
&\ge \int_{\{V \le rK\}} P(x',B)\, P(x,dx') \\
&\ge \alpha_r \nu(B)\, P(x,\{V \le rK\}) \\
&\ge \alpha_r \nu(B)\, \big[1 - P(x,\{V > rK\})\big] \\
&\ge \alpha_r (1 - 1/r)\, \nu(B).
\end{align*}
Thus $P^2$ satisfies a global minorization condition, which yields the stated total variation bound. For the second bound, define the oscillation of a bounded measurable function $f : \X \to \R$ by
\[
\text{osc}(f) := \sup_{x,y \in \X} |f(x) - f(y)|.
\]
Define the residual kernel
\[
R^2(x,\cdot)
:=
\frac{P^2(x,\cdot) - \alpha_r (1 - 1/r)\,\nu(\cdot)}
{1 - \alpha_r (1 - 1/r)}.
\]
Then $R^2$ is a Markov kernel, and for any bounded $f$,
\[
\text{osc}(R^2 f) \le \text{osc}(f).
\]
Hence, for all $x,y \in \X$,
\begin{align*}
|P^2 f(x) - P^2 f(y)|
&= \big[1 - \alpha_r (1 - 1/r)\big]\, |R^2 f(x) - R^2 f(y)| \\
&\le \big[1 - \alpha_r (1 - 1/r)\big]\, \text{osc}(f).
\end{align*}
Therefore,
\[
\text{osc}(P^2 f)
\le \big[1 - \alpha_r (1 - 1/r)\big]\, \text{osc}(f).
\]
Iterating,
\[
\text{osc}(P^{2t} f)
\le \big[1 - \alpha_r (1 - 1/r)\big]^t \text{osc}(f).
\]
Now let $\psi : \X \to \R$ with $|\psi| \le c(V)$. 
By Jensen's inequality, for $x \in \X$
\[
\P |\psi|(x) 
\le \int_{\X} c(V(y)) \P(x, dy)
\le c\left( \int_{\X} V(y) \P(x, dy) \right)
\le c( K ).
\]
Then $| P\psi |$ is bounded by the drift condition, and
\[
\sup_{x \in \X} |P\psi(x)| \le c( K ).
\]
Thus,
\begin{align*}
\text{osc}(P^{2t+1}\psi)
\le \big[1 - \alpha_r (1 - 1/r)\big]^t \text{osc}(P\psi) 
&\le \big[1 - \alpha_r (1 - 1/r)\big]^t 2 \sup_{x} |P\psi(x)| \\
&\le 2 c(K) \, \big[1 - \alpha_r (1 - 1/r)\big]^t.
\end{align*}
Finally, for any $x \in \X$,
\begin{align*}
\left| P^{2t+1}\psi(x) - \Pi(\psi) \right|
&= \left| \int_{\X} \big(P^{2t+1}\psi(x) - P^{2t+1}\psi(y)\big)\, \Pi(dy) \right| \\
&\le \int_{\X} \left| P^{2t+1}\psi(x) - P^{2t+1}\psi(y) \right| \Pi(dy) \\
&\le \text{osc}(P^{2t+1}\psi) \\
&\le 2 c(K) \, \big[1 - \alpha_r (1 - 1/r)\big]^t,
\end{align*}
which completes the proof.
\end{proof}

%The uniform drift \eqref{eq:drift} and local minorization \eqref{eq:local_minorization} imply a contraction in total variation and ensures the existence and uniqueness of the invariant measure $\Pi$.

%There are many results showing geometric convergence for Markov chains in total variation \citep{hairer:mattingly:2011, rosenthal:1995}. We compare the convergence of Theorem~\ref{thm:main} to \citep[Theorem 1.3]{hairer:mattingly:2011}. 

\noindent
Theorem~\ref{thm:main} yields intuitive and explicit convergence bounds under a uniform drift condition and local minorization condition. Since similar conclusions can also be obtained from the classical Harris framework of \cite{hairer:mattingly:2011}, it is natural to compare the resulting bounds. 
In particular, we show that Theorem~\ref{thm:main} yields both improved constants for unbounded test functions and, in many practically relevant settings, sharper convergence rates. 
In \citep{hairer:mattingly:2011}, the proof does not directly follow for a uniform drift condition \eqref{eq:drift} and uses classical geometric drift conditions of the form
\[
\P V(x) \le \gamma V(x) + K
\]
with $\gamma \in (0, 1)$.
In definition \eqref{eq:local_minorization}, it is assumed there is a probability measure $\nu$ and for every $r > 1$, a corresponding sublevel set $\{x \in \X : V(x) \le r K\}$ with minorization constant $\alpha_r \in (0, 1]$.
The minorization condition defined in \citep{hairer:mattingly:2011} uses a smaller minorization constant $\alpha_{2r} \le \alpha_{r}$ where the local minorization condition is over a larger sublevel set with $R_r  = 2 r K / (1 - \gamma)$ and
\[
\inf_{\{x \in \X : V(x) \le R_r \}} P(x, \cdot) \;\ge\; \alpha_{2r/(1-\gamma)} \, \nu(\cdot).
\]
%Here we are using a slight abuse of notation with $\alpha_{2r}$ even though $\gamma > 0$.
Define $\beta_r = a / K$ with $a \in (0, \alpha_{2r/(1-\gamma)})$.
Under conditions \eqref{eq:drift} and \eqref{eq:local_minorization}, then \citep[Theorem 1.3]{hairer:mattingly:2011}, gives an explicit convergence rate $\rho_r \in (0, 1)$ such that,
\begin{align*}
\left|
\P^{t + 1} \phi(x) - \int_{\X} \phi d\Pi
\right|
&\le \rho_r^t \left( 2 + \beta_r PV(x) + \beta_r \int_{\X} V d\Pi \right)
\\
&\le \rho_r^t \left( 2 + \beta_r 2K \right)
\end{align*}
holds uniformly for every $x \in \X$ and all $\phi : \X \to \R$ satisfying $| \phi(y) - \phi(x) | \le 2 + \beta_r [ V(x) + V(y) ]$.
The first comparison concerns the constants appearing in the control of unbounded test functions. For functions satisfying $|\phi| \le 1 + V$, the multiplicative constant in the bound of \cite[Theorem~1.3]{hairer:mattingly:2011} is of order at least $1/\beta_r \ge K /  \alpha_{2r/(1-\gamma)}$, which can be extremely large when $\alpha_{2r}$ is small. By contrast, Theorem~\ref{thm:main} yields the uniform constant $2(K+1)$.
%The first advantage is when controlling unbounded functions $|\phi| \le 1 + V$, the constant is of order larger than $1/\beta_r = K / \alpha_{2r}$ and can be extremely large in comparison to a constant of $2 (K + 1)$ from Theorem~\ref{thm:main}.

The second comparison concerns the convergence rate. 
The explicit form of the convergence rate $\rho_{r}$ from \citep[Theorem 1.3]{hairer:mattingly:2011} is
\[
\rho_r 
= (1 - \alpha_{2r} + a) \vee \left\{ \frac{2 + \frac{a R_r \gamma_0}{K} }{2 + \frac{a R_r}{K}} \right\}
\]
where $\gamma_0 = \gamma + 2 K / R_r$.
We can find a explicit form of the optimal convergence rate $\rho^*_{r}$ by minimizing over the parameter $a \in (0, \alpha_{2r/(1-\gamma)})$ achieving its minimium when 
\begin{align*}
&(1 - \alpha_{2r/(1-\gamma)} + a) = \frac{2 + \frac{a R_r \gamma_0}{K} }{2 + \frac{a R_r}{K}}
\\
&f(a) = a^2 + a (1 - \alpha_{2r/(1-\gamma)} - \gamma) - \frac{2 \alpha_{2r/(1-\gamma)} K }{R_r} = 0.
\end{align*}
Since $f(0) < 0$ and $f(\alpha_{2r/(1-\gamma)}) > 0$ and the quadratic equation has only one positive root, this implies the solution satisfies the required constraint by the intermediate value theorem.
Plugging in the solution of $f(a) = 0$ gives the optimal rate in \citep[Theorem 1.3]{hairer:mattingly:2011} with
\begin{align}
\rho^*_r 
&= 1- \alpha_{2r/(1-\gamma)} + \frac{\gamma + \alpha_{2r/(1-\gamma)} - 1 + \sqrt{ (1 - \alpha_{2r/(1-\gamma)} - \gamma)^2 + \frac{8 \alpha_{2r/(1-\gamma)} K }{R_r}  }}{2}
\\
&= \frac{1 - \alpha_{2r/(1-\gamma)} + \gamma + \sqrt{ (1 - \alpha_{2r/(1-\gamma)} - \gamma)^2 + 4 \alpha_{2r/(1-\gamma)} (1 - \gamma) / r  }}{2}.
\label{eq:optimal_HM_rate}
\end{align}
Taking the limit as $\gamma \to 0$, then $\alpha_{2r/(1- \gamma)}$ should behave approximately like $\alpha_{2r}$ and choosing $r = 1.1$, we can understand when
\[
\sqrt{1 - \alpha_{r} (1 - 1/r)} < \rho^*_r.
\]
Figure~\ref{figure:comparison} compares $\alpha_r$ against $\alpha_{2r}$ and shows the boundary condition.
Figure~\ref{figure:comparison} shows that if the minorization constant scales in such a way that $\alpha_{r} \ge 2 \alpha_{2r}$, then the convergence rate in Theorem~\ref{thm:main} is uniformly smaller.
This is satisfied when $\alpha_{r} = C/r^a$ for $C > 0$ and $a \ge 1$ and so Theorem~\ref{thm:main} can provide a sharper estimate to the convergence rate for minorization scaling worse than $1/r$.
This shows that the convergence rate in Theorem~\ref{thm:main} is strictly smaller depending on how the minorization constant scales from $\alpha_{r}$ to $\alpha_{2r}$ with the increasing size of the sublevel set.
This is often expected to hold in statistics applications as the minorization conditions often exhibit exponentially poor scaling.

\begin{figure}
\centering
\includegraphics[width=.8\linewidth]{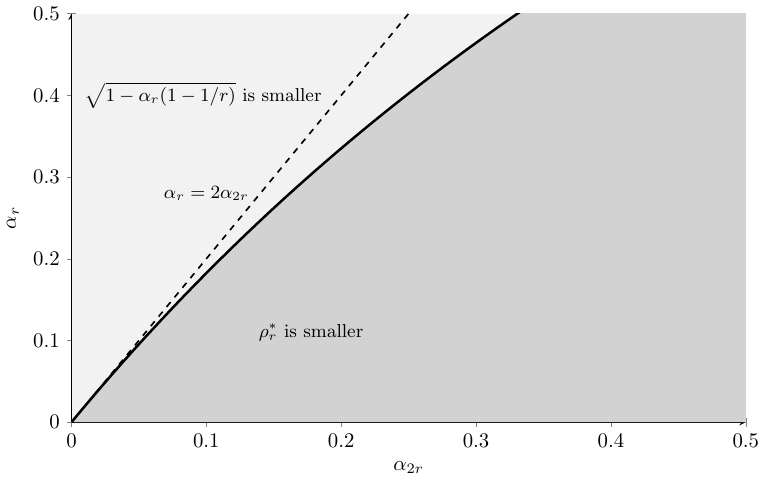}
\caption{
Comparison of convergence rates from Theorem~\ref{thm:main} and the optimal convergence rate $\rho_r^*$ from \citep[Theorem 1.3]{hairer:mattingly:2011}. The parameter $r = 1.1$ is used.
}
\label{figure:comparison}
\end{figure}

\section{Optimal hyper-V uniform convergence rates} \label{sec:optimal:rates}

Theorem~2.1 provides simple and explicit hyper-$V$ uniform ergodicity bounds that clearly exhibit intuition and the role of the drift and minorization parameters. While these bounds are often sufficient in practice, they are not generally optimal. In this section, we derive sharp rates by analyzing the interaction between visits to the minorization region and its complement. The resulting rates are explicit, optimal within the present framework, and can substantially improve upon the bounds of Theorem~2.1.

\begin{theorem}[Optimal hyper-$V$ uniform ergodicity]  \label{thm:convergence_sharp}
Assume the drift condition \eqref{eq:drift} holds with a Borel measurable function $V : \X \to [0, \infty)$ and constant $K \in (0,\infty)$ and the minorization condition \eqref{eq:local_minorization} holds with a constant $\alpha_r \in (0, 1]$ and Borel probability measure $\nu$ on $\X$.
Let $\gamma_r = 1 - 1/r$, and define
\[
\lambda_{\pm}(\gamma_r, \alpha_r)
= \frac{1 - \gamma_r \wedge \alpha_r \pm \sqrt{ (1 - \gamma_r \wedge \alpha_r)^2 + 4 ( \gamma_r \wedge \alpha_r ) (1 - \alpha_r \vee \gamma_r) } }{2}.
\]
Let $c : [0, \infty) \to [0, \infty)$ be any nondecreasing concave function.
Then for all $x \in \X$ and for all integers $t \ge 0$,
\begin{align*}
&\left\| P^{t}(x, \cdot) - \Pi(\cdot) \right\|_{\mathrm{TV}}
\le \rho_t(\gamma_r, \alpha_r)
\\
&\left\| P^{t + 1}(x, \cdot) - \Pi(\cdot) \right\|_{c(V)}
\le 2 c(K) \rho_{t}(\gamma_r, \alpha_r)
\end{align*}
where
\begin{align}
\rho_{t}(\gamma_r, \alpha_r)
= \frac{(1 - \lambda_{-}(\gamma_r, \alpha_r)) \lambda_{+}(\gamma_r, \alpha_r)^t - (1 - \lambda_{+}(\gamma_r, \alpha_r)) \lambda_{-}(\gamma_r, \alpha_r)^t }{ \lambda_{+}(\gamma_r, \alpha_r) - \lambda_{-}(\gamma_r, \alpha_r) }
\label{eq:sharp_rate}
\end{align}
and $\lim_{t \to \infty} \rho_{t}(\gamma_r, \alpha_r)^{1/t} = \lambda_{+}(\gamma_r, \alpha_r)$.
\end{theorem}

\begin{proof}
It suffices to complete the proof assuming $\alpha_r \in (0, 1)$.
For positive integers $t$ and Borel sets $A \subseteq \X$, define $d_t(A) = \sup_{x, y \in A} \norm{\P^t(x, \cdot) - \P^t(y, \cdot)}_{\text{TV}}$.
Define the set $C_r = \{ x \in \X : V(x) \le  r K \}$. 
By Markov's inequality,
\[
P(x, C_r^c)
\le \frac{PV(x)}{r K}
\le \frac{K}{r K}
= 1- \gamma_r,
\qquad \forall x \in \X.
\]
This implies that 
\[
P(x, C_r) \ge \gamma_r.
\]
Let $x \in \X$ and write $\mu_x = \P(x, \cdot)$.
Then we can define the conditional measure $\mu_{x, C_r} = \P(x, \cdot \cap C_r) / \P(x, C_r)$.
For every Borel $B \subseteq \X$,
\[
\mu_x(B) - \gamma \mu_{x, C_r}(B)
= \mu_{x, C_r}(B) \mu_x(C_r) + \mu_x(B \cap C_r^c) - \gamma \mu_{x, C_r}
\ge 0.
\]
The residual probability measure is well defined where
\[
R_x = \frac{\mu_x - \gamma \mu_{x, C_r}}{1 - \gamma_r}
\]
and $\P(x, \cdot) = \gamma_r \mu_{x, C_r} + (1 - \gamma_r) R_x$.
For $x, y \in \X$, we have the simple identity
\[
\P^{t + 1}( x, \cdot ) - \P^{t + 1}( y, \cdot )
= \mu_{x} \P^{t} - \mu_{y} \P^{t}.
\]
Using the residual measure, we then have
\begin{align*}
\P^{t + 1}( x, \cdot ) - \P^{t + 1}( y, \cdot )
= \gamma_r ( \mu_{x, C_r} \P^{t} - \mu_{y, C_r} \P^{t} ) + ( 1 - \gamma_r ) ( R_x \P^{t} - R_y \P^{t} ).
\end{align*}
Taking the supremum over $x, y$, we conclude that
\[
d_{t + 1}(\X)
\le \gamma_r d_{t}(C_r) + ( 1 - \gamma_r ) d_{t}(\X).
\]
Using the minorization condition \eqref{eq:local_minorization}, for $x \in C_r$, we can then define a residual probability measure
\[
S_x(\cdot) = \frac{ \P(x, \cdot) - \alpha_r \nu(\cdot) }{ 1 - \alpha_r }.
\]
Then for all $x, y \in C_r$,
\begin{align*}
\left| \P^{t + 1}( x, \cdot ) - \P^{t + 1}( y, \cdot ) \right|
&\le ( 1 - \alpha_r ) \left| S_x \P^{t} - S_y \P^{t} \right|.
\end{align*}
Using the drift condition \eqref{eq:drift},
\[
S_x(C_r^c)
\le \frac{\P V(x)}{ r K  (1 - \alpha_r)}
\le \frac{1}{ r  (1 - \alpha_r)}
\le \frac{1 - \gamma_r}{1 - \alpha_r}.
\]
So 
\[
S_x(C_r)
\ge \frac{(\gamma_r - \alpha_r)_+}{1 - \alpha_r}.
\]
Let $\beta_r = (\gamma_r - \alpha_r)_+$ denoting the positive part.
Applying the previous argument used with $\mu_x$, then for all $x, y \in C_r$,
\begin{align*}
( 1 - \alpha_r ) \left| S_x \P^{t} - S_y \P^{t} \right|
\le ( 1 - \alpha_r ) \left( \frac{\beta_r}{1 - \alpha_r} d_{t}(C_r) + \left( 1 - \frac{\beta_r}{1 - \alpha_r} \right) d_{t}(\X) \right).
\end{align*}
Therefore,
\begin{align*}
d_{t+1}(C_r) &\le \beta_r d_{t}(C_r) + \left( 1 - \alpha_r - \beta_r \right) d_{t}(\X).
\end{align*}
Define the matrix 
\[
M_r = \begin{pmatrix}
1 - \gamma_r & \gamma_r \\
1 - \alpha_r - \beta_r  & \beta_r
\end{pmatrix}.
\]
Then with the inequality understood coordinate-wise:
\begin{align*}
\begin{pmatrix}
d_{t+1}(\X) \\
d_{t+1}(C_r)
\end{pmatrix}
&\le M_r \begin{pmatrix}
d_{t}(\X) \\
d_{t}(C_r)
\end{pmatrix}
\end{align*}
Since $M_r$ has nonnegaitve entries, then we apply this recursively to get
\begin{align*}
\begin{pmatrix}
d_{t}(\X) \\
d_{t}(C_r)
\end{pmatrix}
\le M_r^{t} \begin{pmatrix}
d_{0}(\X) \\
d_{0}(C_r)
\end{pmatrix}.
\end{align*}
Define $\rho_{r, t} = (1, 0) M_r^t (1, 1)^T$.
Then using matrix multiplication $d_{t}(\X) \le \rho_{r, t}$.
We have then shown that for every $x, y \in \X$,
\[
\norm{\P^t(x, \cdot) - \P^t(y, \cdot)}_{\text{TV}}
\le \rho_{r, t}.
\]
Therefore, integrating with respect to $\Pi$ concludes
\[
\sup_{x \in \X} \norm{\P^t(x, \cdot) - \Pi(\cdot)}_{\text{TV}}
\le \rho_{r, t}.
\]

We now extend this to uniform convergence in the weighted $V$-norm. 
Define the oscillation of a bounded measurable function $f : \X \to \R$ by
\[
\text{osc}(f) := \sup_{x,y \in \X} |f(x) - f(y)|.
\]
Now let $\psi : \X \to \R$ with $|\psi| \le c(V)$. 
By Jensen's inequality, for $x \in \X$
\[
\P |\psi|(x) 
\le \int_{\X} c(V(y)) \P(x, dy)
\le c\left( \int_{\X} V(y) \P(x, dy) \right)
\le c( K ).
\]
Then $| P\psi |$ is bounded by the drift condition, and
\[
\sup_{x \in \X} |P\psi(x)| \le c( K ).
\]
Thus,
\begin{align*}
\text{osc}(P^{t+1}\psi)
\le \rho_{r, t}  2 \sup_{x} |P\psi(x)|
\le \rho_{r, t}  2 c(K).
\end{align*}
Therefore, integrating with respect to $\Pi$
\[
\sup_{x \in \X} \norm{\P^{t + 1}(x, \cdot) - \Pi(\cdot)}_{c(V)}
\le \rho_{r, t} 2 c(K).
\]

It remains to show an explicit form of $\rho_{r, t}$ that is represented by formula \eqref{eq:sharp_rate}.
First assume $\gamma_r > \alpha_r$.
The characteristic polynomial of $M_r$ is defined by
\[
p(\lambda)
= \lambda^2 - \tr(M_r) \lambda + \det(M_r)
=  \lambda^2 - (1 - \alpha_r) \lambda - \alpha_r (1 - \gamma_r).
\]
Next, assume $\gamma_r \le \alpha_r$.
The characteristic polynomial in this case is
\[
p(\lambda)
= \lambda^2 - \tr(M_r) \lambda + \det(M_r)
= \lambda^2 - (1 - \gamma_r) \lambda - \gamma_r (1 - \alpha_r).
\]
Therefore, it suffices to complete the proof assuming $\gamma_r \le \alpha_r$. 
In this case, the characteristic polynomial has roots
\[
\lambda_{\pm} = \frac{1 - \gamma_r \pm \sqrt{ (1 - \gamma_r)^2 + 4 \gamma_r (1 - \alpha_r) } }{2}.
\]
Using the Cayley-Hamilton theorem
\[
M_R^2 = (1 - \gamma_r) M_r + \gamma_r (1 - \alpha_r) I
\]
where $I$ is the identity matrix.
So then multiplying by the row vector $(1, 0)$,
\begin{align*}
&M_r^{t + 2} = (1 - \gamma_r) M_r^{t + 1} + \gamma_r (1 - \alpha_r) M_r^t
\\
&\rho_{r, t + 2} = (1 - \gamma_r) \rho_{r, t + 1} + \gamma_r (1 - \alpha_r) \rho_{r, t}.
\end{align*}
Now, similarly both roots satisfy this relation and for $a \in [0, \infty)$
\[
a \lambda_+^t + (1 - a) \lambda_-^t
\]
satisfies the relation as well.
We can solve directly that $\rho_{r, 0} = \rho_{r, 1} = 1$.
We must solve $a$ to satisfy the initial conditions and then $\rho_{r, t} = a \lambda_+^t + (1 - a) \lambda_-^t$ for all integers $t \ge 0$.
Since $a + 1 - a = 1$, this is true if
$
a \lambda_+ + (1 - a) \lambda_- = 1.
$
This is solved by 
\[
a = \frac{1 - \lambda_-}{ \lambda_+ - \lambda_- }.
\]
Therefore,
\[
\rho_{r, t}
= \frac{(1 - \lambda_-) \lambda_+^t + (\lambda_+ - 1) \lambda_-^t }{ \lambda_+ - \lambda_- }.
\]
Since $\lambda_+ > -\lambda_-$ using Vieta's formula, $(\lambda_- / \lambda_+)^t \to 0$ as $t \to \infty$, then $\lim_{t \to \infty} \rho_{r, t}^{1/t} = \lambda_+$.
\end{proof}

The technique of analyzing the convergence of Markov chains by lifting to a two-dimensional space has been used previously in \citep{canizo:mischler:2023} and \citep{hairer:mattingly:2011}.
The main advantage with this new convergence result is the sublevel set for the minorization condition is always smaller.
We show now the upper bound and convergence rate in Theorem~\ref{thm:convergence_sharp} are optimal in a minimax sense over all kernels satisfying the uniform drift \eqref{eq:drift} and minorization conditions \eqref{eq:local_minorization}.
For $R, K \in (0, \infty)$ with $R > K$ and $\alpha \in (0, 1]$, let $\mathcal{F}(R, \alpha, K)$ be the set of $(\X, \P)$ of all state spaces $\X$ with Borel Markov kernels $\P$ on $\X$ with invariant measure $\Pi_\P$ satisfying
\begin{align*}
&\sup_{x \in \X} P V(x) \le K,
&\inf_{\{x \in \X : V(x) \le R \}} P(x, \cdot) \;\ge\; \alpha \, \nu(\cdot)
\end{align*}
for some Borel function $V : \X \to [0, \infty)$ and some Borel probability measures $\nu$ on $\X$.
We are interested in finding the optimal upper bound at each integer $t \ge 0$, the smallest upper bound $u \ge 0$ such that
\begin{align}
\sup_{\P \in \mathcal{F}(R, \alpha, K)} \sup_{x \in \X} \norm{ \P^t(x, \cdot) - \Pi_{\P} }_{\text{TV}} \le u.
\label{eq:optimal_ub}
\end{align}

\begin{proposition}
Let $R, K \in (0, \infty)$ with $R > K$ and $\alpha \in (0, 1]$.
For all integers $t \ge 0$,
\[
\sup_{\P \in \mathcal{F}(R, \alpha, K)} 
\sup_{x \in \X} \norm{ \P^t(x, \cdot) - \Pi_{\P} }_{\text{TV}}
= \rho_t(1 - K/R, \alpha)
\]
and $\rho_t(1 - K/R, \alpha)$ is the optimal upper bound in \eqref{eq:optimal_ub}.
The convergence rate $\lambda_{+}(1 - K/R, \alpha)$ is optimal in the sense that
\[
\sup_{\P \in \mathcal{F}(R, \alpha, K)} 
\lim_{t \to \infty} \left( \sup_{x \in \X} \norm{ \P^t(x, \cdot) - \Pi_{\P} }_{\text{TV}} \right)^{1/t}
= \lambda_{+}(1 - K/R, \alpha).
\]
\end{proposition}

\begin{proof}
Theorem~\ref{thm:convergence_sharp} implies the upper bounds and it will suffice to show the lower bounds.
Consider the state space $\{ 1, -1, 2, -2, 3 \}$.
Let $\e \in (0, 1)$ and define $V(1) = V(-1) = R + \e$ and $V(2) = V(-2) = V(3) = 0$.
Then the sublevel set $\{ V \le R \} = \{ 2, -2, 3 \}$.
Define $\gamma_\e = 1 - K / (R + \e)$.
We define a Markov kernel with $\sigma \in \{-1, 1\}$ by
\begin{align*}
&\P_\e(\sigma 1, \cdot) = (1 - \gamma_\e) \delta_{\sigma 1} + \gamma_\e \delta_{\sigma 2}
\\
&\P_\e(\sigma 2, \cdot) = \alpha \delta_{3} + ( 1 - \alpha - (\gamma_\e - \alpha)_+ ) \delta_{\sigma 1} + (\gamma_\e - \alpha)_+ \delta_{\sigma 2}
\\
&\P_\e(3, \cdot) = \delta_{3}.
\end{align*}
By construction, we can see $\delta_{3}$ is invariant.
Then $\P_\e V(3) = 0$ and
\begin{align*}
&\P_\e V(\sigma 1) = (1 - \gamma_\e) (R + \e) = K
\\
&\P_\e V(\sigma 2) = ( 1 - \alpha - (\gamma_\e - \alpha)_+ ) (R + \e)
\le K.
\end{align*}
We also have that
$
\inf_{ \{ x \in \X : V(x) \le R \} } \P_\e(x, \cdot)
\ge \alpha \delta_3.
$
Therefore, the drift and minorization conditions hold.
So the dynamics moving between states $\{+ 1, + 2 \}$ without entering state $\{ 3 \}$ are described by the matrix
\[
M_e
=
\begin{pmatrix}
1 - \gamma_\e & \gamma_\e \\
1 - \alpha - (\gamma_\e - \alpha)_+ & (\gamma_\e - \alpha)_+
\end{pmatrix}.
\]
So then we can lower bound the total variation using the fact that
\[
\P_{\e}^t(+1, \{+1, +2 \})
\ge (1, 0) M_{\e}^t (1, 1)^T.
\]
and $\delta_{3}(\{+1, +2\}) = 0$.
We then have
\[
\norm{ \P_\e^t(+1, \cdot) - \delta_{3} }_{\text{TV}}
\ge (1, 0) M_{\e}^t (1, 1)^T.
\]
Using continuity and taking the limit 
\[
\liminf_{\e \downarrow 0} (1, 0) M_{\e}^t (1, 1)^T 
\ge (1, 0) M_{R}^t (1, 1)^T 
\ge \rho_t(\gamma_R, \alpha).
\]
So then
\begin{align*}
\sup_{\mathcal{F}(R, \alpha, K)} 
\liminf_{t \to \infty} \left( \sup_{x \in \X} \norm{ \P^t(x, \cdot) - \Pi  }_{\text{TV}} \right)^{1/t}
&\ge \liminf_{t \to \infty} \left( \norm{ \P_\e^t(+1, \cdot) - \Pi  }_{\text{TV}} \right)^{1/t}.
\end{align*}
The proof is completed taking the limit $\e \downarrow 0$.
\end{proof}

These results conclude the convergence bound in Theorem~\ref{thm:convergence_sharp} is optimal.
In particular, the upper bound constant can be much larger with poor scaling in previous results \citep{hairer:mattingly:2011}.
Since the minorization is over a smaller set, the optimal convergence rate from \ref{thm:convergence_sharp} is a strict improvement in the interesting cases when the minorization constant is small and $\alpha_{2r} < \alpha_{r}$ from \eqref{eq:optimal_HM_rate}.

The weighted-$V$ bounds obtained in the preceding result can be transferred to Wasserstein bounds by relating the underlying metric to the Lyapunov function $V$. Existing results on Wasserstein ergodicity (see, e.g., \cite{hairer:mattingly:2011, HairerMattinglyScheutzow2011, Eberle2016}) typically use a form of Wasserstein contractivity of the Markov kernel—either globally or on a suitable subset of the state space—together with a Lyapunov drift condition. 
By contrast, the present framework replaces such contractivity assumptions with a uniform drift condition and an associated local minorization, yielding Wasserstein convergence without imposing explicit Lipschitz or coupling conditions on the transition kernel.
In certain contexts, a local minorization condition is impossible to establish and the proofs in this section must be extended to hold for a local Wasserstein contraction.
However, the present framework can be a useful tool to establish uniform ergodicity equivalents in Wasserstein distances.

The Wasserstein distance is defined for Borel probability measures on a complete separable metric space $\X$ and a lower semicontinuous metric $d : \X \times \X \to [0, \infty]$ by minimizing the expected distances over all joint distributions so that
\[
W_d\left( \P^t(x, \cdot) , \Pi \right)
= \inf\left\{ \int_{\X \times \X} d(x, y) \Gamma(dx, dy) : \Gamma \in \C(\P^t(x, \cdot), \Pi(\cdot) ) \right\}
\]
where $\C(\P^t(x, \cdot), \Pi(\cdot) )$ denotes the set of all Borel probability measures on $\X \times \X$ satisfying $\Gamma(\cdot \times \X) = \P^t(x, \cdot)$ and $\Gamma(\X \times \cdot) = \Pi(\cdot)$.

\begin{corollary}[Wasserstein extension] \label{cor:wasserstein}
Let assumptions \eqref{eq:drift} with a Borel measurable function $V : \X \to [0, \infty)$ and \eqref{eq:local_minorization} hold and $\alpha_r, \gamma_r$ as defined in Theorem~\ref{thm:convergence_sharp}.
Suppose further that there exists a metric $d$ on $\X$ such that $(\X, d)$ is a complete separable metric space and for some $x_0 \in \X$ and some concave nondecreasing function $c : [0, \infty) \to [0, \infty)$,
\[
d(x,x_0) \le c(V(x)).
\]
Then for all $x \in \X$ and all integers $t \ge 0$,
\[
W_d (P^{t + 1}(x,\cdot), \Pi)
\le 2 c(K) \rho_{t}(\gamma_r, \alpha_r).
\]
\end{corollary}

\begin{proof}
    Suppose $\psi: \X \to \mathbb{R}$ is a Borel measurable function such that $|\psi(x) - \psi(y)| \leq d(x,y)$ for all $x,y \in \X$. It follows that 
    $$
    |\psi(x) - \psi(x_0)| \leq d(x,x_0) \leq c(V(x)). 
    $$

    \noindent
    Let $\tilde{\psi}: \X \to \mathbb{R}$ be defined by $\tilde{\psi}(x) = \psi(x) - \psi(x_0)$ for every $x \in \X$. By Theorem \ref{thm:convergence_sharp}, it follows that 
$$
\left|
P^{t} \tilde{\psi}(x) - \int_{\X} \tilde{\psi} \, d\Pi
\right|
\le 2 c(K) \rho_{t-1} 
$$

\noindent
for every $t \geq 1$. Since $P^{t}(x, \cdot)$ and $\Pi(\cdot)$ are probability measures, it follows that 
$$
\left|
P^{t} {\psi}(x) - \int_{\X} \psi \, d\Pi
\right|
\le 2 c(K) \rho_{t-1}
$$

\noindent
The result now follows by noting that 
$$
W_d \left( P^{t}(x, \cdot), \Pi(\cdot) \right) = \sup_{\psi: \X \to \mathbb{R} {\mbox{ \small{\tiny{ such that} } } } |\psi(x) - \psi(y)| \leq d(x,y) \; \forall x,y \in \X} \left|
P^{t} {\psi}(x) - \int_{\X} \psi \, d\Pi
\right|. 
$$
Here we used that since $\X$ is a complete separable metric space and $d$ is continuous, the Wasserstein distance admits a dual representation as the supremum over Lipschitz functions by Kantorovich-Rubinstein duality \citep{kantorovich:rubinshtein:1958, villani:2008}.
\end{proof}

\section{Examples}
\label{section:examples}

\noindent
In this section, we illustrate the applicability of the proposed framework through several examples of Markov chain Monte Carlo algorithms on general state spaces. In each case, we focus on the associated Markov transition kernel and verify that it satisfies a uniform $V$-drift condition together with a local minorization condition. This allows us to establish hyper-$V$ uniform ergodicity with explicit rates, while avoiding the need for global minorization arguments used in previous analyses.

\subsection{The P\'olya--Gamma sampler for logistic regression} \label{example:pg_sampler}

\noindent
We consider the P\'olya--Gamma Gibbs sampler introduced by \cite{PolsonScottWindle2013} in the context of Bayesian logistic regression. Let $y = (y_1,\dots,y_n)$ denote binary responses taking values in $\{0, 1 \}$ with corresponding covariates $x_1,\dots,x_n \in \R^p$, and let $X \in \R^{n \times p}$ denote the design matrix with rows $x_i^\top$. Let $\beta \in \R^p$ denote the regression coefficients, and assume a Gaussian prior $\beta \sim N(m_0, \Sigma_0)$, where $N(m_0, \Sigma_0)$ denotes the $p$-dimensional Gaussian distribution with mean $m_0 \in \R^p$ and symmetric covariance matrix $\Sigma_0 \in \R^{p \times p}$. We will assume that $\Sigma_0$ is positive definite. 

To sample from the intractable posterior $\Pi(\cdot \mid X, y)$ (obtained by combining the logistic regression likelihood with the prior above), the P\'olya--Gamma augmentation introduces latent variables $\omega = (\omega_1,\dots,\omega_n)$ such that, conditional on $\beta$, the variables $\omega_i$ are independent with
\[
\omega_i \mid \beta \sim \mathrm{PG}(1, x_i^\top \beta), \quad i = 1,\dots,n,
\]
where $\mathrm{PG}(b,c)$ denotes the P\'olya--Gamma distribution supported on $(0,\infty)$ with parameters $b > 0$ and $c \in \R$ introduced in \cite{PolsonScottWindle2013}. The resulting Gibbs sampler alternates between updating $\omega$ and $\beta$, and induces a Markov chain $\{\beta^{(t)}\}_{t \ge 0}$ on $\R^p$ obtained by marginalizing out $\omega$. 
Let $P_{PG}$ denote the Markov transition kernel.

The one-step transition of this Markov chain (with Markov transition kernel denoted by $P_{\mathrm{PG}}(\cdot, \cdot)$) from $\beta^{(t)}$ to $\beta^{(t+1)}$ can be described as follows. Given $\beta^{(t)}$, sample
\[
\omega_i^{(t+1)} \mid \beta^{(t)} \sim \mathrm{PG}(1, x_i^\top \beta^{(t)}), 
\quad i = 1,\dots,n,
\]
independently, and define $\Omega^{(t+1)} = \mathrm{diag}(\omega_1^{(t+1)},\dots,\omega_n^{(t+1)})$. Then sample
\[
\beta^{(t+1)} \mid \omega^{(t+1)} \sim N(m_{\omega^{(t+1)}}, \Sigma_{\omega^{(t+1)}}),
\]
where
\[
\Sigma_{\omega} = \left(X^\top \Omega X + \Sigma_0^{-1}\right)^{-1}, 
\qquad
m_{\omega} = \Sigma_{\omega} \left(X^\top \kappa + \Sigma_0^{-1} m_0 \right),
\]
with $\kappa = y - \tfrac{1}{2}\mathbf{1}_n$. The result below shows that the marginal Markov chain $\{\beta^{(t)}\}_{t\ge 0}$ induced by the P\'olya--Gamma Gibbs sampler satisfies a uniform drift condition and a local minorization condition, and is therefore hyper-$V$ uniformly ergodic. In addition, the resulting bounds extend to the Wasserstein distance.

\begin{theorem} \label{thm:pg_convergence}
The P\'olya--Gamma chain $\{\beta^{(t)}\}_{t\ge 0}$ described above is hyper-$V$ uniformly ergodic for $V(\beta) = \|\beta\|^2$. In particular, if $P_{\mathrm{PG}}$ denotes the corresponding Markov transition kernel and $\rho_{t}$ is defined by \eqref{eq:sharp_rate}, then for all $\beta \in \R^p$ and integers $t \ge 1$,
\begin{align*}
\| P_{\mathrm{PG}}^{t}(\beta,\cdot) - \Pi(\cdot \mid X, y) \|_{\mathrm{TV}}
&\le 
\rho_{t}(\gamma_r, \e_r)
\\
\| P_{\mathrm{PG}}^{t + 1}(\beta,\cdot) - \Pi(\cdot \mid X, y) \|_{1 + V}
&\le 2 (1 + K) \rho_{t}(\gamma_r, \e_r)
\end{align*}
where for $r > 1$, $\gamma_r = 1 - 1/r$
\begin{align*}
&K = \|\Sigma_0\|^2 \cdot \|X^\top \kappa + \Sigma_0^{-1} m_0\|^2
\;+\; \tr(\Sigma_0),
&\e_r = \prod_{i=1}^n \big[ \cosh( \|x_i\|\sqrt{rK} /2 ) \big]^{-1}.
\end{align*}
Moreover, if $W_{ \ell_2 }$ denotes the Wasserstein distance induced by the $\ell_2$-metric on $\mathbb{R}^p$, then 
\[
W_{ \ell_2 }\big(P_{\mathrm{PG}}^{t+1}(\beta,\cdot), \Pi(\cdot \mid X, y) \big)
\le  2 (1 + K) \rho_{t}(\gamma_r, \e_r),
\qquad \forall \beta \in \R^p,\; t \ge 1.
\]
\end{theorem}

\begin{proof}
Define $V(\beta) = \|\beta\|^2$. Then for all $\beta \in \R^p$,
\begin{align*}
\P_{\mathrm{PG}} V(\beta)
= \E\big[ V(\beta^{(t+1)}) \mid \beta^{(t)} = \beta \big] 
= \E\big[ \|\beta^{(t+1)}\|^2 \mid \beta^{(t)} = \beta \big].
\end{align*}
Since $\beta^{(t+1)} \mid \omega^{(t+1)} \sim N(m_{\omega^{(t+1)}}, \Sigma_{\omega^{(t+1)}})$, we have
\[
\E\big[ \|\beta^{(t+1)}\|^2 \mid \omega^{(t+1)} \big]
= \|m_{\omega^{(t+1)}}\|^2 + \tr(\Sigma_{\omega^{(t+1)}}).
\]
Therefore,
\begin{align*}
\P_{\mathrm{PG}} V(\beta)
&\le \E\big[ \|m_{\omega^{(t+1)}}\|^2 + \tr(\Sigma_{\omega^{(t+1)}}) \mid \beta^{(t)} = \beta \big] \\
&= \E\big[ \|\Sigma_{\omega^{(t+1)}} (X^\top \kappa + \Sigma_0^{-1} m_0)\|^2 
+ \tr(\Sigma_{\omega^{(t+1)}}) \mid \beta^{(t)} = \beta \big].
\end{align*}

\noindent
Now observe that
\[
\Sigma_{\omega} = \big(X^\top \Omega X + \Sigma_0^{-1}\big)^{-1}
\preceq \Sigma_0,
\]
since $X^\top \Omega X$ is positive semidefinite. Consequently 
$\|\Sigma_{\omega}\| \le \|\Sigma_0\|$ and $\tr(\Sigma_{\omega}) 
\le \tr(\Sigma_0)$, for all $\omega$. It follows that
\begin{align*}
\P_{\mathrm{PG}} V(\beta)
&\le \|\Sigma_0\|^2 \cdot \|X^\top \kappa + \Sigma_0^{-1} m_0\|^2
\;+\; \tr(\Sigma_0) =: K, 
\end{align*}

\noindent
and hence $\P_{\mathrm{PG}} V$ is uniformly bounded. 

Let $p_{\mathrm{PG}}(\beta, \beta')$ denote the transition density of the marginal chain. It admits the representation
\[
p_{\mathrm{PG}}(\beta, \beta')
= \int p(\beta' \mid \omega)\, p(\omega \mid \beta)\, d\omega,
\]
where $\beta' \mid \omega \sim N(m_\omega, \Sigma_\omega)$ and $p(\omega \mid \beta) = \prod_{i=1}^n p(\omega_i \mid \beta)$ with $\omega_i \mid \beta \sim \mathrm{PG}(1, x_i^\top \beta)$, for $i=1,\dots,n$. Let $C := \{\beta \in \R^p : V(\beta) \le rK\}$. Then $\|\beta\|\le \sqrt{rK}$, so
\[
|x_i^\top \beta| \le \|x_i\| \|\beta\| \le \|x_i\|\sqrt{rK} =: c_i. 
\]
Let $f_{PG}(\cdot \mid a, b)$ denote the P\'olya--Gamma density with parameters $a$ and $b$. Using the form of the P\'olya--Gamma density (see \cite[Section 2]{PolsonScottWindle2013}) for all $|t|\le c_i$ and $\omega>0$,
\[
f_{PG}(\omega_i \mid 1,t) = \frac{\cosh(t/2) e^{\frac{c_i^2-t^2}{2} \omega_i}}{\cosh(c_i/2)}\, f_{PG}(\omega \mid 1,c_i) \ge \frac{1}{\cosh(c_i/2)}\, f_{PG}(\omega_i \mid 1,c_i).
\]
Hence for all $\beta \in C$,
\[
p(\omega \mid \beta)
= \prod_{i=1}^n f_{PG} (\omega_i \mid 1, x_i^\top \beta)
\ge \epsilon \prod_{i=1}^n f_{PG} (\omega_i \mid 1,c_i),
\quad
\epsilon := \prod_{i=1}^n \big(\cosh(c_i/2)\big)^{-1}.
\]

\noindent
Therefore, for any Borel set $B \subset \R^p$,
\begin{align*}
P_{\mathrm{PG}}(\beta, B)
&= \int_B \int p(\beta' \mid \omega)\, p(\omega \mid \beta)\, d\omega\, d\beta' \\
&\ge \epsilon \int_B \int p(\beta' \mid \omega)\, \prod_{i=1}^n f_{PG} (\omega_i \mid 1, c_i) \, d \omega\, d\beta'.
\end{align*}
Defining
\[
\nu(B) := \int_B \int p(\beta' \mid \omega)\, \prod_{i=1}^n f_{PG}(\omega_i \mid 1, c_i) \, d \omega \, d\beta',
\]
we obtain 
\[
P_{\mathrm{PG}} (\beta, B) \ge \epsilon \, \nu(B),
\quad \beta \in C,
\]
which establishes the local minorization condition. Let $\beta_0 \in \mathbb{R}^p$ denote the vector with all entries equal to zero. Then, for any $\beta \in \mathbb{R}^p$, it follows that 
$$
\|\beta - \beta_0\| = \|\beta\| \leq 1 + \|\beta\|^2 = 1 + V(\beta). 
$$

\noindent
The uniform Wasserstein convergence now follows from Corollary \ref{cor:wasserstein}. 
\end{proof}

We are interested in the dependence on $\e_r$ as a function of the radius $r$ of the sublevel set.
We simulate logistic regression on synthetic data of size $n = 100$ and $p = 10$ repeating the simulation independently $100$ times.
We choose the prior parameters $m_0 = 0$ and $\Sigma_0 = \sigma^2 I_p$ with $\sigma^2 \in \{ .5, 1, 1.5, 2 \}$.
Figure~\ref{figure:pg_sim} plots the dependence of $r$ on $\log(\e_r)$ from $1 < r < 2$ and $\e(r)$ for a small interval of $r$.
The dashed line shows a reference comparison of the scaling with the width of the minorization set radius.
The simulations show that scaling of the minorization constant can rapidly decrease where Theorem~\ref{thm:main} and the optimal convergence rate in Theorem~\ref{thm:convergence_sharp} can provide a substantial benefit over previous convergence results.
Simulations were performed using Python with packages NumPy and Matplotlib.

\begin{figure}[t]
\centering
\begin{subfigure}{.49\linewidth}
  \centering
  \includegraphics[width=\linewidth]{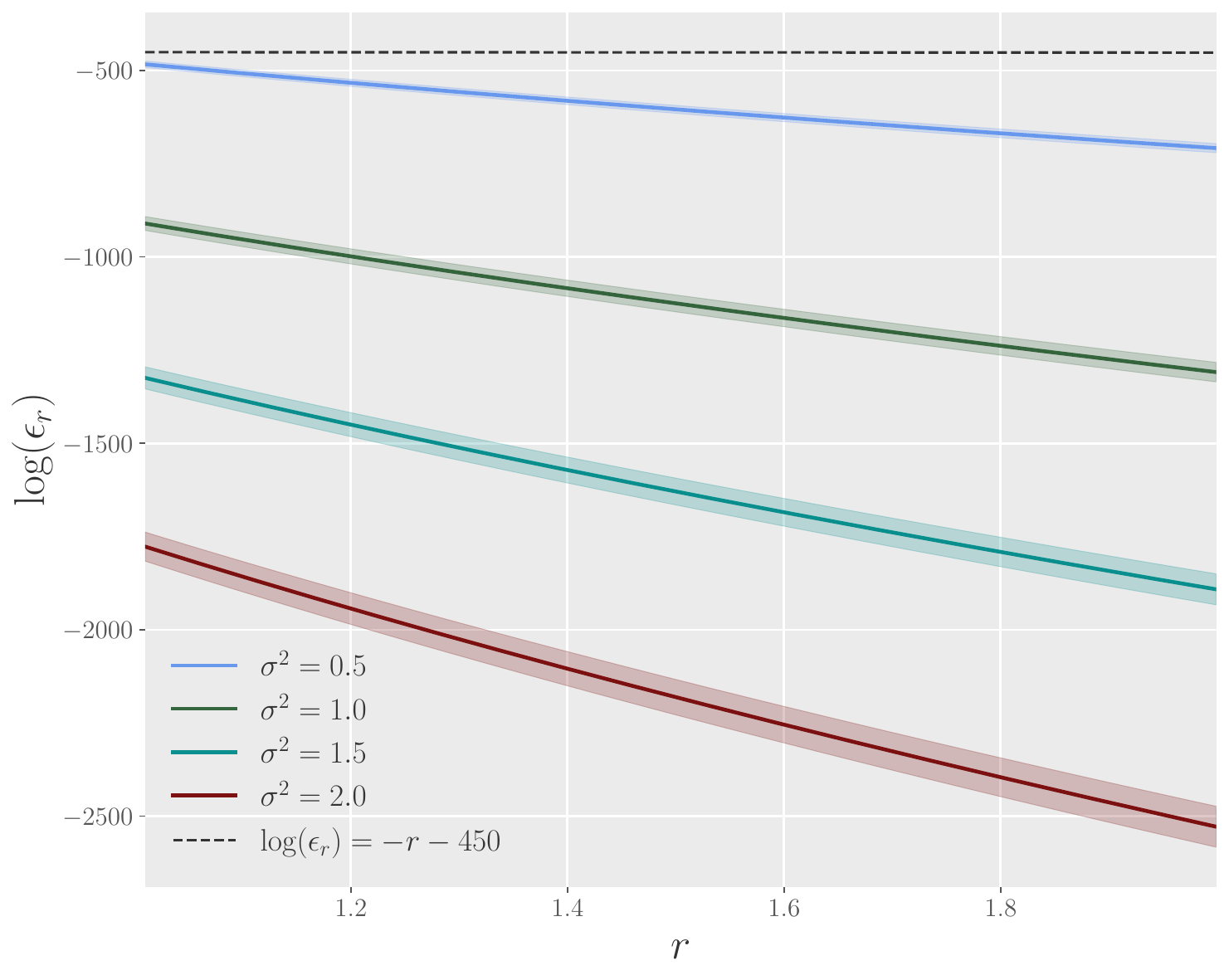}
  \caption{}
\end{subfigure}
\begin{subfigure}{.49\linewidth}
  \centering
  \includegraphics[width=\linewidth]{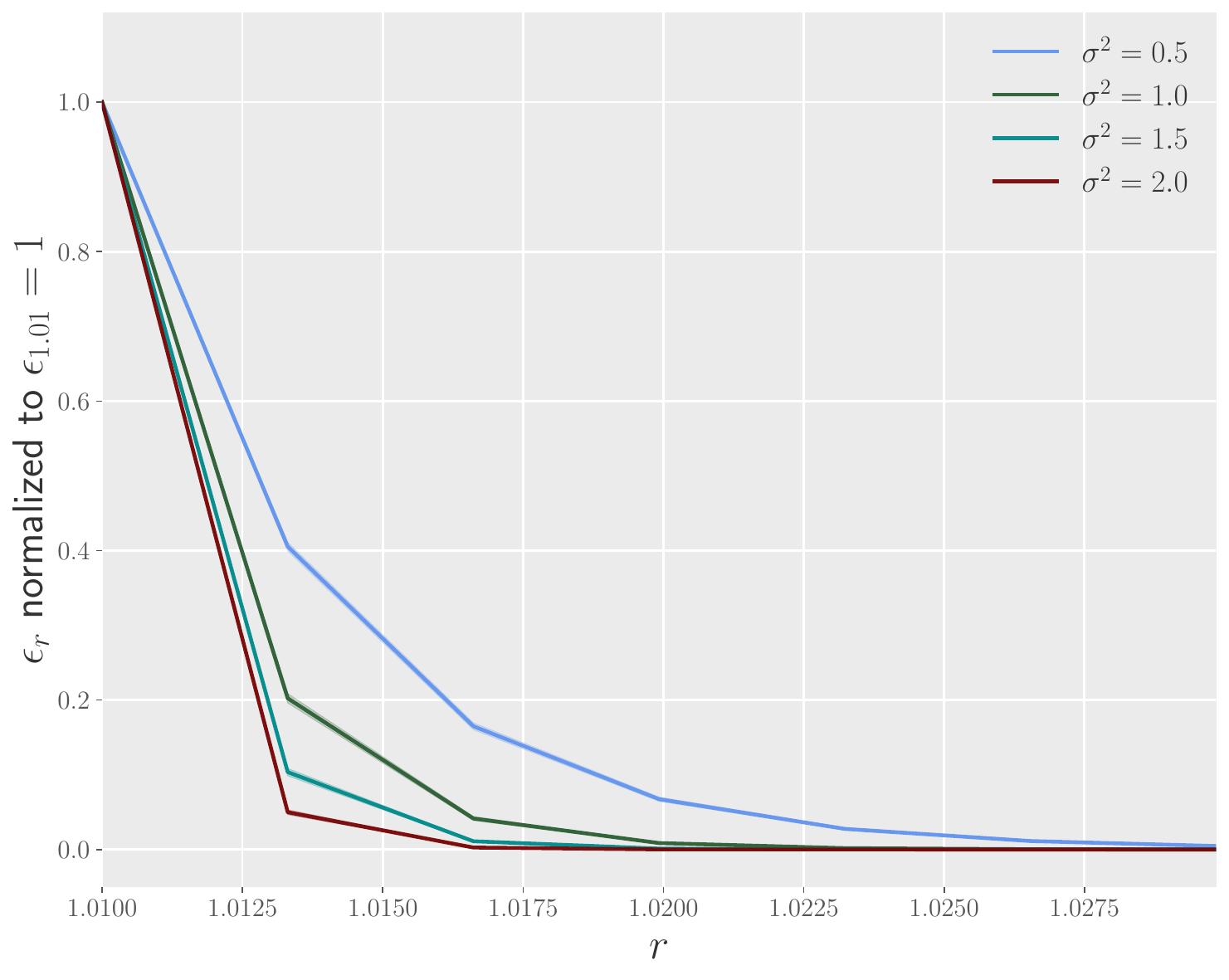}
  \caption{}
\end{subfigure}
\caption{
Logistic regression on synthetic data of size $n = 100$ and $p = 10$ repeated $100$ times with prior parameters $m_0 = 0$ and $\Sigma_0 = \sigma^2 I_p$ with $\sigma^2 \in \{ .5, 1, 1.5, 2 \}$.
Figure (a) shows the dependence of $r$ from $1 < r < 2$ on $\log(\e(r))$ with shaded regions representing 3 standard deviations.
Figure (b) shows the dependence of $r$ on $\e_r$ on a small scale of $r$ and standardized so that $\e_{1.01} = 1$ with shaded regions representing 3 standard deviations.
} 
\label{figure:pg_sim}
\end{figure}

\subsection{The Kolmogorov--Gamma sampler for continuous proportion regression} \label{example:kg_sampler}

\noindent
We next consider the Kolmogorov--Gamma Gibbs sampler in the context of Bayesian regression with continuous proportion data, as proposed by \cite{LeeEtAlKG}. Let $y = (y_1,\dots,y_n)$ denote observations taking values in $(0,1)$, with corresponding covariates $x_1,\dots,x_n \in \R^p$ and design matrix $X \in \R^{n \times p}$. Let $\beta \in \R^p$ denote the regression coefficients, with prior $\beta \sim N(m_0, \Sigma_0)$.

The regression model in \cite{LeeEtAlKG}, combined with the prior above, leads to an intractable posterior $\Pi_{KG} (\cdot \mid X, y)$. To sample from this intractable posterior distribution, \cite{LeeEtAlKG} introduce auxiliary variables $\omega = (\omega_1,\dots,\omega_n)$ such that, conditional on $\beta$, the variables $\omega_1,\dots,\omega_n$ are independent with distributions of the form 
\[
\omega_i \mid \beta, y \sim \mathrm{KG}(\lambda, x_i^\top \beta), \quad i = 1,\dots,n,
\]
where $\mathrm{KG}(\cdot,\cdot)$ denotes the Kolmogorov--Gamma (KG) distribution arising from the augmentation scheme of \cite{LeeEtAlKG}, and $\lambda > 0$ will be assumed to be known.
As in the P\'olya--Gamma setting, we refer to \cite{LeeEtAlKG} for a full description of the KG distribution and invoke its properties only as needed in the sequel. 
This construction induces a Markov chain $\{\beta^{(t)}\}_{t \ge 0}$ on $\R^p$ via marginalization over $\omega$. 

The one-step transition of this Markov chain (with Markov transition kernel denoted by $P_{\mathrm{KG}}(\cdot, \cdot)$) from $\beta^{(t)}$ to $\beta^{(t+1)}$ can be described as follows. Given $\beta^{(t)}$, sample
\[
\omega_i^{(t+1)} \mid \beta^{(t)} \sim \mathrm{KG}(\lambda, x_i^\top \beta^{(t)}),
\quad i = 1,\dots,n,
\]
independently, and define $\Omega^{(t+1)} = \mathrm{diag}(\omega_1^{(t+1)},\dots,\omega_n^{(t+1)})$. Then sample
\[
\beta^{(t+1)} \mid \omega^{(t+1)} \sim N(m_{\omega^{(t+1)}}, \Sigma_{\omega^{(t+1)}}),
\]
where
\[
\Sigma_{\omega} = \left(X^\top \Omega X + \Sigma_0^{-1}\right)^{-1}, 
\qquad
m_{\omega} = \Sigma_{\omega} \left(X^\top z + \Sigma_0^{-1} m_0 \right),
\]
and $z \in \R^n$ is a working response depending on $y$ through the Kolmogorov--Gamma augmentation scheme. As in the P\'olya--Gamma case, the result below shows that the marginal Markov chain $\{\beta^{(t)}\}_{t\ge 0}$ satisfies a uniform drift condition and a local minorization condition, and is therefore hyper-$V$ uniformly ergodic. In addition, the resulting bounds extend to the Wasserstein distance.

\begin{theorem} \label{thm:kg_convergence}
The Kolmogorov--Gamma chain $\{\beta^{(t)}\}_{t\ge 0}$ described above is hyper-$V$ uniformly ergodic for $V(\beta) = \|\beta\|^2$. In particular, if $P_{\mathrm{KG}}$ denotes the corresponding transition kernel and $\rho_{t}$ is defined by \eqref{eq:sharp_rate}, then for all $\beta \in \R^p$ and integers $t \ge 1$,
\begin{align*}
\| P_{\mathrm{KG}}^{t}(\beta,\cdot) - \Pi_{KG}(\cdot \mid X, y) \|_{\mathrm{TV}}
&\le \rho_{t}(\gamma_r, \e_r), \\
\| P_{\mathrm{KG}}^{t + 1}(\beta,\cdot) - \Pi_{KG}(\cdot \mid X, y) \|_{1 + V}
&\le 2 (1 + K) \rho_{t}(\gamma_r, \e_r),
\end{align*}
where for any $r > 1$, $\gamma_r = 1 - 1/r$ and
\begin{align*}
&c_i = \|x_i\| \sqrt{r K},
&\epsilon_r = \left( \prod_{i=1}^n \frac{c_i/2}{\sinh(c_i/2)} \right)^\lambda,
\\
&K = \|\Sigma_0\|^2 \|X^\top z + \Sigma_0^{-1} m_0\|^2
\;+\; \tr(\Sigma_0).
\end{align*}
Moreover, if $W_{ \ell_2 }$ denotes the Wasserstein distance induced by the $\ell_2$-metric, then
\[
W_{ \ell_2 }\big(P_{\mathrm{KG}}^{t+1}(\beta,\cdot), \Pi_{KG}(\cdot \mid X, y) \big)
\le 2 ( 1 + K ) \rho_{t}(\gamma_r, \e_r),
\qquad \forall \beta \in \R^p,\; t \ge 1.
\]
\end{theorem}

\begin{proof}
We follow the same strategy as in the P\'olya--Gamma case, and highlight only the essential modifications. Let $V(\beta) = \|\beta\|^2$. Then for all $\beta \in \R^p$,
\[
P_{\mathrm{KG}} V(\beta)
= \E\big[ \|\beta^{(t+1)}\|^2 \mid \beta^{(t)} = \beta \big]. 
\]
Since $\beta^{(t+1)} \mid \omega \sim N(m_\omega, \Sigma_\omega)$, it follows that 
\[
P_{\mathrm{KG}} V(\beta)
\le \E\big[ \|m_\omega\|^2 + \tr(\Sigma_\omega) \mid \beta \big].
\]
Using $\Sigma_\omega = (X^\top \Omega X + \Sigma_0^{-1})^{-1} \preceq \Sigma_0$, we obtain $\|\Sigma_\omega\| \le \|\Sigma_0\|$ and $\tr(\Sigma_\omega) \le \tr(\Sigma_0)$, and therefore
\[
P_{\mathrm{KG}} V(\beta)
\le \|\Sigma_0\|^2 \|X^\top z + \Sigma_0^{-1} m_0\|^2
\;+\; \tr(\Sigma_0)
=: K.
\]
This establishes the uniform drift condition. For the local minorization condition, let
\[
p_{\mathrm{KG}}(\beta,\beta')
= \int p(\beta' \mid \omega)\, p(\omega \mid \beta)\, d\omega,
\]
and define $C := \{\beta \in \R^p : V(\beta) \le rK\}$. Then $\|\beta\|\le \sqrt{rK}$, so
\[
|x_i^\top \beta|
\le \|x_i\| \|\beta\|
\le \|x_i\| \sqrt{rK}
=: c_i.
\]

\noindent
Let $f_{KG}(\cdot \mid b, c)$ denote the density of the $KG(b,c)$ distribution. Using the explicit form of the Kolmogorov--Gamma density (see Theorem~2 of \cite{LeeEtAlKG}), for all $|t|\le c_i$, we have 
\[
f_{KG}(\omega_i \mid \lambda, t)
=
\left( \frac{\sinh(t/2)/(t/2)}{\sinh(c_i/2)/(c_i/2)} \right)^\lambda
\exp\!\left(\frac{c_i^2 - t^2}{2}\omega_i\right)
f_{KG}(\omega_i \mid \lambda, c_i).
\]
Since the function $t \mapsto \sinh(t/2)/(t/2)$ is an even function (on $\mathbb{R}$), increasing on $\mathbb{R}_+$, and satisfies $\sinh(t/2)/(t/2) \ge 1$ (with $\sinh(0)/0$ defined as $1$), it follows that
\[
f_{KG}(\omega_i \mid \lambda, t)
\ge \left( \frac{c_i/2}{\sinh(c_i/2)} \right)^\lambda
\, f_{KG}(\omega_i \mid \lambda, c_i).
\]

Thus for all $\beta \in C$,
\[
p(\omega \mid \beta)
= \prod_{i=1}^n f_{KG}(\omega_i \mid \lambda, x_i^\top \beta)
\ge \epsilon \prod_{i=1}^n f_{KG}(\omega_i \mid \lambda, c_i),
\quad
\epsilon := \left( \prod_{i=1}^n \frac{c_i/2}{\sinh(c_i/2)} \right)^\lambda. 
\]
Proceeding exactly as in the P\'olya--Gamma case, define
\[
\nu(B) := \int_B \int p(\beta' \mid \omega)
\, \prod_{i=1}^n f_{KG}(\omega_i \mid \lambda, c_i)\, d\omega\, d\beta',
\]
so that
\[
P_{\mathrm{KG}}(\beta, B)
\ge \epsilon \, \nu(B),
\quad \beta \in C,
\]
establishing the local minorization condition. Finally, since $\|\beta\| \le 1 + V(\beta)$, the Wasserstein convergence follows from Corollary~\ref{cor:wasserstein}. 
\end{proof}

\subsection{Independence Metropolis-Hastings sampler} \label{example:imh}

The previous examples showed that a uniform drift condition \eqref{eq:drift} combined with a local minorization \eqref{eq:local_minorization} is easier to verify than a global minorization.
However, this example shows this is not always the case.
Consider the Metropolis-Hastings Markov chain with independent proposal.
Let $Q$ be a Borel probability measure on $\X$, and we will assume $dQ/d\Pi$ exists and is finite on $\X$.
Let $\delta_a$ denote the Dirac measure at $a \in \X$.
Let $\wedge, \vee$ denote the minimum and maximum respectively.
Define the Metropolis-Hastings Markov transition kernel with proposal $Q$ for $x \in \X$ by 
\[
\P_{IMH}(x, dy)
= a(x, y) Q(dy)
+ \delta_x(dy) \int_{\X} \left[ 1 - a(x, y) \right] Q(dy)
\]
where the acceptance function is defined
\[
a(x, y) = 1 \wedge \left( \frac{ \frac{dQ}{d\Pi}(x) }{ \frac{dQ}{d\Pi}(y) } \right).
\]
In order to be geometrically ergodic, we must have $\inf_{x \in \X} dQ/d\Pi(x) > 0$ and in this case, we have a contraction in total variation yielding uniform ergodicity \citep{mengersen:tweedie:1996}.
The following shows a drift condition \eqref{eq:drift}.

\begin{proposition}
\label{proposition:imh_drift}
Define the drift function $V = dQ/d\Pi$ and assume $\int_{\X} V dQ < \infty$.
Then for every $x \in \X$,
\[
\P_{IMH} V(x)
\le \int_{\X} V(y) Q(dy).
\]
\end{proposition}

\begin{proof}
For every $x \in \X$,
\begin{align*}
\P_{IMH} V(x)
&= \int_{\X} V(y) \wedge V(x) Q(dy)
+ V(x) \int_{\X} \left[ 1 - 1 \wedge \left( \frac{ V(x) }{ V(y) } \right) \right] Q(dy) 
\\
&= \int_{\X} V(y) \wedge V(x) V(y) \Pi(dy)
+ V(x) \int_{\X} \left[ 1 - 1 \wedge \left( \frac{ V(x) }{ V(y) } \right) \right] V(y) \Pi(dy). 
\end{align*}
Using $2 ab \le a^2 + b^2$ for $a, b \ge 0$,
\begin{align*}
\P_{IMH} V(x)
&= \int_{\X} V(y) \wedge V(x) V(y) \Pi(dy)
+ \int_{ \{ y : V(y) \ge V(x) \} } \left[ V(x) V(y) - V(x)^2 \right] \Pi(dy)
\\
&\le \int_{ \{ y : V(y) \ge V(x) \} } V(x) V(y) \Pi(dy)
+ \int_{ \{ y : V(y) < V(x) \} } V(y)^2 \Pi(dy)
\\
&\hspace{.4cm}+ \int_{ \{ y : V(y) \ge V(x) \} } \left[ V(y)^2 - V(x) V(y) \right] \Pi(dy)
\\
&\le \int_{\X} V(y)^2 \Pi(dy)
\\
&= \int_{\X} V(y) Q(dy).
\end{align*}
\end{proof}

On any sublevel sets of $V$, a local minorization \ref{eq:local_minorization} will require the global infimum bounded below:
\[
\inf_{x \in \X} (dQ/d\Pi)(x) > 0.
\] 
The importance weight must still be uniformly bounded corresponding to previous results \citep{mengersen:tweedie:1996}.
In this case, Proposition~\ref{proposition:imh_drift} does not provide any benefit to show convergence in standard total variation using Theorem~\ref{thm:main}.

Consider a concrete toy example with $\pi(x) = \exp(-x) I_{\{ x \ge 0 \} }$ and $\lambda \in (0, 1]$, then the proposal $q(x) = \lambda \exp(-\lambda x) I_{ \{ x \ge 0 \} }$.
Then $q(x)/\pi(x) = \lambda \exp[( 1 - \lambda )x] \ge \lambda$ for all $x \in [0, \infty)$.
However the drift condition requires
\[
\int_{[0, \infty)} \frac{q(x)}{\pi(x)} q(x) dx
= \lambda^2 \int_{[0, \infty)} \exp[(1 - 2\lambda) x] dx < \infty
\]
and also requires $\lambda \in (1/2, 1]$.
Importantly, when $\lambda \in (0, 1/2)$ this drift condition does not hold, but the global minorization condition holds for every choice of $\lambda$.
The uniform drift condition here provides a stronger uniform convergence over functions $f$ such that $|f| \le V$ that can be more useful in statistics applications.

\section{Qualitative results for general data augmentation Gibbs samplers} \label{example:general_data_augmentation}

We now investigate qualitative results of general two-variable Gibbs samplers constructed via data augmentation.
Let $S_p$ be the space of symmetric positive semidefinite matrices in $\R^{p \times p}$.
Let $\mu$ be a Borel probability measure supported on $\Omega$.
Let $F : \Omega \to S_p$ be Borel measurable and let $\kappa : \Omega \to \R^p$ be Borel measurable.
Assume a target measure $\Pi$ on the joint space $\R^p \times \Omega$ is defined by
\begin{align}
\Pi(d\beta, d\w)
\propto \exp\left( \beta^T \kappa(\w) -\frac{1}{2} \beta^T \left[ F(\w) + \Sigma_0^{-1} \right] \beta \right) \mu(d\w) d\beta
\label{eq:general_posterior_form}
\end{align}
where $\Sigma_0 \in \R^{p \times p}$ is symmetric and positive definite.
Here we are implicitly assuming $\Pi$ defines a valid probability measure with finite normalizing constant.
This formulation of the target measure is prescribing $\Pi(d\beta \mid \w)$ as Gaussian with a specific form.

Define a two-variable Gibbs sampler $(\w^{(t)}, \beta^{(t)})_{t \ge 0}$ with invariant measure \eqref{eq:general_posterior_form} described as follows.
Given $\beta^{(t)}$, sample
\begin{align}
&\w^{(t+1)} | \beta^{(t)} \sim \Pi(d\w^{(t+1)} \mid \beta^{(t)}).
\end{align}
Then given $\w^{(t+1)}$, sample
\begin{align}
&\beta^{(t+1)} | \w^{(t+1)} \sim N_p(m(\w^{(t+1)}, \Sigma(\w^{(t+1)}))
\label{eq:general_gibbs_sampler_form}
\end{align}
where
\begin{align*}
&\Sigma(\cdot) = \left( F(\cdot) + \Sigma_0^{-1} \right)^{-1},
&m(\cdot) = \Sigma(\cdot) \kappa(\cdot).
\end{align*}

Let $\P_G$ denote the marginal Markov transition kernel for the marginal Markov chain $(\beta^{(t)})_{t \ge 0}$.
The following result shows that qualitative hyper V-uniform ergodicity from Theorem~\ref{thm:main} can be inferred directly from the form of the target measure without any convergence analysis.

\begin{theorem} \label{thm:general_posterior_convergence}
Assume $\sup_{\w \in \Omega} \norm{ \Sigma(\w) \kappa(\w) } < \infty$.
Let
\[
Z(\beta) = \int_{\Omega} \exp\left( \beta^T \kappa(\w) -\frac{1}{2} \beta^T \left[ F(\w) + \Sigma_0^{-1} \right] \beta \right) \mu(d\w)
\]
and assume $\beta \mapsto Z(\beta)$ is finite and continuous.
Then the marginal Markov chain $(\beta^{(t)})_{t \ge 0}$ described above is hyper $V$-uniformly ergodic for $V(\beta) = \norm{\beta}^2$, that is, there is an $\alpha \in (0, 1)$ and $K > 0$ such that for all $\beta \in \R^p$ and all integers $t \ge 1$,
\begin{align*}
&\norm{P_G^{2t}(\beta, \cdot) -  \Pi(\cdot \times \Omega)}_{\mathrm{TV}}
\le \big(1 - \alpha \big)^t
\\
&\norm{P_G^{2t+1}(\beta, \cdot) -  \Pi(\cdot \times \Omega)}_{1 + V}
\le 2( 1 + K) \, \big(1 - \alpha \big)^t.
\end{align*}
\end{theorem}

\begin{proof}
By the definition of the target measure \eqref{eq:general_posterior_form}, define the unnormalized density
\begin{align*}
\tilde{\pi}(\beta, \w)
&= \exp\left( 
\beta^T \kappa(\w) 
- \frac{1}{2} \beta^T \left[ F(\w) + \Sigma_0^{-1} \right] \beta 
\right)
\\
&= \exp\left( 
\beta^T \kappa(\w) 
- \frac{1}{2} \beta^T \Sigma(\w)^{-1} \beta 
\right).
\end{align*}
Both $\Sigma(\w)$ and $\Sigma(\w)^{-1}$ are well-defined due to the positive definite assumption on $\Sigma_0$.
Therefore, we can define the Gibbs sampler $(\w^{(t)}, \beta^{(t)})_{t \ge 0}$ through the conditional updates $\w^{(t+1)} | \beta^{(t)} \sim \pi(\w^{(t+1)} \mid \beta^{(t)})$ and
\[
\beta^{(t+1)} | \w^{(t+1)} \sim N_p(m(\w^{(t+1)}), \Sigma(\w^{(t+1)})).
\]

We will show using Theorem~\ref{thm:main} that the marginal of the Gibbs sampler is uniformly ergodic.
We first prove the drift condition as in \eqref{eq:drift}.
Let $(\lambda_i(\Sigma(\w)))_{i = 1}^p, (\lambda_i(\Sigma_0))_{i = 1}^p$ denote the eigenvalues of $\Sigma(\w)$ and $\Sigma_0$ respectively. 
Then $\lambda_i(\Sigma(\w)) \le \lambda_i(\Sigma_0)$, and we have
\begin{align*}
\E\left[  \norm{ \beta^{(t+1)} }^2 \mid \w^{(t+1)} = \w \right]
\le \norm{m(\w)}^2 + \tr(\Sigma(\w))
\le \norm{m(\w)}^2 + \tr(\Sigma_0).
\end{align*}

Next, we prove the local minorization condition \eqref{eq:local_minorization}.
The sublevel sets of $V$ are compact since they are bounded in Euclidean space.
For $\beta \in \R^p$, let $\pi(\cdot \mid \beta)$ denote the conditional density with respect to $\mu$.
Since $Z(\beta)$ is finite and continuous by assumption, then $\pi(\w \mid \beta)$ is positive and continuous for every $\w \in \Omega$.
For any $R > 0$, define
\[
g_R(\cdot) = \inf_{\norm{\beta}^2 \le R} \pi(\cdot \mid \beta ).
\]
By continuity and compactness of the sublevel set, $g_R$ is Borel measurable and positive.
For all Borel measurable $B \subseteq \R^p$ and $A \subseteq \Omega$
\[
\int_{B} \int_{A} \Pi(d\beta' \mid \w') \Pi(d\w' \mid \beta)
\ge \int_{B} \int_{A} \Pi(d\beta' \mid \w') g_R(\w') \mu(d\w').
\]
Define $M_R = \int_{\Omega} g_R(\w') \mu(d\w')$ and it is clear this is finite, and define the probability measure $\nu_R$ by
\[
\nu_R(B)
= \frac{1}{M_R} \int_{B} \int_{\Omega} \Pi(d\beta' \mid \w') g_R(\w')  \mu(d\w').
\]
Therefore, we have shown a local minorization with
\[
\inf_{\norm{\beta}^2 \le R} \int_{B} \int_{\Omega} \Pi(d\beta' \mid \w') \Pi(d\w' \mid \beta)
\ge M_R \nu_R(B).
\]
\end{proof}

Theorem~\ref{thm:general_posterior_convergence} gives a qualitative V-uniform ergodicity result for many posteriors in statistics.
A sufficient and readily checked condition is that $\sup_{\w \in \Omega} \norm{\kappa(\w)} < \infty$.
In this case, $\sup_{\w \in \Omega} \norm{ \Sigma(\w) \kappa(\w) } \le \norm{ \Sigma_0 } \norm{ \kappa(\w) } < \infty$.
Since for any $\e > 0$, we can upper bound $2 \beta^T \sup_{\w \in \Omega} \kappa(\w) \le \e \norm{\beta}^2 + \e^{-1} \sup_{\w \in \Omega} \norm{ \kappa(\w) }^2$, then by dominated convergence, $Z(\cdot)$ is continuous.
The following are some concrete examples.

\begin{example}
(P\'olya-Gamma Gibbs sampler)
Consider the P\'olya--Gamma posterior for Bayesian logistic regression with Gaussian prior in Example~\ref{example:pg_sampler}. 
The conditional posterior has Lebesgue density of the form
\begin{align*}
\pi(\beta \mid \w, X, y )
\propto \exp\left( \beta^T \kappa -\frac{1}{2} \beta^T \left[ \Sigma_0^{-1}  + X^T \Omega X \right] \beta \right) 
\end{align*}
where $\kappa = X^T ( y - \tfrac{1}{2}\mathbf{1}_n ) + \Sigma_0^{-1} m_0$ and $\Omega = \text{diag}(\w_1, \ldots, \w_n)$.
Since $\kappa$ is uniformly bounded, then Theorem~\ref{thm:general_posterior_convergence} gives a qualitative hyper V-uniform ergodicity but without an explicit convergence rate.
\end{example}

\begin{example}
(Kolmogorov-Gamma Gibbs sampler)
Consider the posterior from Kolmogorov--Gamma data augmentation in Example~\ref{example:kg_sampler}.
This is defined for $(\w_i, \lambda_i)_{i = 1}^n$ where $\lambda_i \in \{1, \ldots, L \}$ with $L < \infty$ is a mixture distribution and $\w |\lambda_i$ have distribution $KG(\lambda_i, x_i^T \beta)$. 
This results in an augmented conditional posterior Lebesgue density of the form
\begin{align*}
\pi( \beta \mid \w, \lambda, X, y )
\propto \exp\left( \beta^T \kappa(\lambda) -\frac{1}{2} \beta^T \left[ \Sigma_0^{-1} + X^T \Omega X \right] \beta \right) 
\end{align*}
where $\kappa(\lambda) = X^T ( \lambda \cdot (y - \tfrac{1}{2}\mathbf{1}_n) ) + \Sigma_0^{-1} m_0$ and $\Omega = \text{diag}(\w_1, \ldots, \w_n)$.
Similar to the P\'olya-Gamma sampler, the form of the posterior and Theorem~\ref{thm:general_posterior_convergence} directly gives a qualitative hyper V-uniform ergodicity without the explicit convergence rate.
\end{example}

\begin{example}
(Qualitative convergence for Bayesian linear regression)
Uniform drift conditions such as \eqref{eq:drift} have been shown in Bayesian regression with independent priors \citep[Lemma 4.1]{Ekvall:Jones:2021}.
We prescribe this model in a general way covering a wide range of applications in statistics.
For $v \in (m-1, \infty)$ and positive definite symmetric matrix $S \in \R^{m \times m}$, we consider an inverse-Wishart prior $\Sigma \sim \text{Wis}_m^{-1}(v, S)$ with density 
\begin{align*}
g_{v, S}(\Sigma) = \frac{\det(S)^{\frac{v}{2}}}{2^{vm/2} \Gamma_m(v/2)} \det(\Sigma)^{-\frac{v + m + 1}{2}} \exp(-\frac{1}{2} \tr(S \Sigma^{-1}))
\end{align*}
where $\Gamma_m$ is the multivariate Gamma function.
For $m_0 \in R^{mp}$ and a positive definite matrix $C \in \R^{pm \times pm}$, we consider $\beta = \text{vec}(\B) \sim N_{pm}(m_0, C)$.
Let $Y \in \R^{n \times m}$, $y = \text{vec}(Y)$, and $x_1, \ldots, x_n \in \R^p$ and let $X \in \R^{n \times p}$ with rows $x_i^T$.
Consider the linear regression model
\begin{align*}
&y_i \mid X , \beta \sim N_m( \B^T x_i, \Sigma),
&i = 1, \ldots, n.
\end{align*}
The posterior density has the form
\begin{align*}
&\pi(\Sigma, \beta \mid X, Y)
\\
&\propto \exp\left( -\frac{1}{2} \tr\left\{ \Sigma^{-1} \left[ S + (Y - X \B)^T  (Y - X \B) \right] \right\} 
-\frac{1}{2} (\beta - m_0)^T C^{-1} (\beta - m_0) \right)
\det(\Sigma)^{-\frac{v + n + m + 1}{2}}
\end{align*}
Writing the posterior density in terms of the vectorized $\beta$
\begin{align*}
&\pi(\Sigma, \beta \mid X, Y)
\\
&\propto \exp\left( -\frac{1}{2} (y - (I_m \otimes X) \beta )^T (\Sigma^{-1} \otimes I_n) (y - (I_m \otimes X) \beta ) -\frac{1}{2} (\beta - m_0)^T C^{-1} (\beta - m_0) \right)
\\
&\hspace{.4cm} \times \exp\left( -\frac{1}{2} \tr( \Sigma^{-1} S ) \right) 
\det(\Sigma)^{-\frac{v + n + m + 1}{2}}.
\end{align*}
The conditional posterior has the form as in Theorem~\ref{thm:general_posterior_convergence} with
\begin{align*}
\pi(\beta \mid \Sigma, X, Y)
\propto \exp\left( \beta^T \kappa_{\Sigma}
-\frac{1}{2} \beta^T C_{\Sigma}^{-1} \beta \right)
\end{align*}
where
\begin{align*}
&C_{\Sigma} = \left( C^{-1} + \Sigma^{-1} \otimes X^T X \right)^{-1},
&\kappa_{\Sigma} = \left( C^{-1} m_0 + ( \Sigma^{-1} \otimes X^T ) y \right).
\end{align*}
The two variable Gibbs sampler alternates updates for integers $t \ge 0$
\begin{align*}
&\Sigma^{(t + 1)} \mid \beta^{(t)} \sim \text{Wis}_m^{-1}(v + n, S + (Y - X \B^{(t)})^T (Y - X \B^{(t)}) ),
\\
&\beta^{(t + 1)} \mid \Sigma^{(t + 1)} \sim N_{mp}\left( C_{\Sigma^{(t+1)}} \kappa_{\Sigma^{(t+1)}}, C_{\Sigma^{(t+1)}} \right).
\end{align*}
The normalizing constant of the conditional inverse-Wishart density is known explicitly and is continuous in $\beta$.
If $X^T X$ is invertible, then define the maximum likelihood estimate
\[
\hat{\beta} = (I_m \otimes (X^T X)^{-1}) (I_m \otimes X^T) y.
\]
We then have the upper bound by adding and subtracting $C^{-1} \hat{\beta}$
\begin{align*}
\norm{ C_{\Sigma} \kappa_{\Sigma} }
&= \norm{ C_{\Sigma } C^{-1} ( m_0 - \hat{\beta} ) + \hat{\beta} }
\\
&\le \norm{ C_{ \Sigma } } \norm{ C^{-1} } \norm{ m_0 - \hat{\beta} } + \norm{\hat{\beta} }
\\
&\le \norm{ C } \norm{ C^{-1} } \norm{ m_0 - \hat{\beta} } + \norm{\hat{\beta} }.
\end{align*}
Combined with the form of the posterior, Theorem~\ref{thm:general_posterior_convergence} gives a qualitative hyper V-uniform ergodicity.
Using previous techniques, an explicit convergence rate can be shown through standard convergence analysis \citep[Lemma 4.1]{Ekvall:Jones:2021}.
\end{example}

\section{Summary and future directions} \label{section:conclusion}

We introduced a uniform drift condition \eqref{eq:drift} together with a local minorization condition \eqref{eq:local_minorization} and showed that they imply a stronger weighted form of uniform ergodicity, which we call hyper-$V$ uniform ergodicity. 
%Although the underlying proof techniques in Theorem~\ref{thm:main} are classical, the resulting framework has important practical advantages. 
We showed a minimax optimal convergence bound and convergence rate in Theorem~\ref{thm:convergence_sharp} and the resulting framework has important practical advantages. 
In many statistical applications, establishing a global minorization condition is substantially more difficult than verifying a uniform drift condition and a local minorization condition. This was illustrated by the P\'olya--Gamma and Kolmogorov--Gamma Gibbs samplers, where the required assumptions can be verified through relatively simple drift and density comparison arguments. 

The advantages of the framework are not limited to simplifying the verification of uniform ergodicity. In the independence Metropolis--Hastings example, the assumptions are not necessarily easier to verify than existing approaches, but they nevertheless yield the stronger conclusion of hyper-$V$ uniform ergodicity rather than merely uniform ergodicity. This demonstrates that the framework can provide additional information even when it does not substantially simplify the analysis. 

A natural direction for future research is to replace the local minorization condition with a local Wasserstein contraction. Such an approach could potentially lead to sharper convergence rates while preserving the weighted convergence guarantees developed here. We also have not fully explored the consequences of hyper-$V$ uniform ergodicity in Wasserstein distance, which may warrant further investigation. 

For two-variable Gibbs samplers, we showed that qualitative hyper-$V$ uniform ergodicity can often be inferred directly from the structure of the invariant distribution. This is particularly noteworthy because convergence can be established without any direct analysis of the associated Markov chain. By contrast, it is not generally clear whether a global minorization condition can be deduced from the form of the invariant distribution alone. More broadly, the ideas developed in Section~\ref{example:general_data_augmentation} suggest that the structure of the invariant measure itself may provide a useful route to establishing convergence properties for broader classes of data-augmentation algorithms. Although the resulting conclusions are qualitative, such methods have the potential to become practical tools for applied statisticians seeking to establish uniform ergodicity without undertaking a detailed convergence analysis.

\section*{Statement of AI use} 

GPT 5.6 Sol was used to assist in the development of proof techniques appearing in Section 3, and the resulting proofs were independently verified by the authors. GPT 5.6 Sol was also used to assist with the Python code used for the simulation studies. All other mathematical results, proofs, theoretical developments, and manuscript content were developed by the authors without the use of AI-generated proof assistance. 

%\begin{acks}[Acknowledgments]
%\end{acks}

\begin{funding}
Khare's work on this paper was partially supported by NSF-DMS-2410677 and NSF-DMS-2506059. 
\end{funding}

\bibliographystyle{imsart-number}
\bibliography{references}

\end{document}